\documentclass[12pt,reqno]{article}

\usepackage[usenames]{color}
\usepackage{amssymb}
\usepackage{graphicx}
\usepackage{amscd}

\usepackage[colorlinks=true,
linkcolor=webgreen,
filecolor=webbrown,
citecolor=webgreen]{hyperref}

\definecolor{webgreen}{rgb}{0,.5,0}
\definecolor{webbrown}{rgb}{.6,0,0}

\usepackage{color}
\usepackage{fullpage}
\usepackage{float}

\usepackage{graphics,amsmath,amssymb}
\usepackage{amsthm}
\usepackage{amsfonts}
\usepackage{latexsym}
\usepackage{epsf}

\begin{document}


\theoremstyle{plain}
\newtheorem{thm}{Theorem}
\newtheorem{cor}[thm]{Corollary}
\newtheorem{lem}[thm]{Lemma}
\newtheorem{prop}[thm]{Proposition}

\theoremstyle{definition}
\newtheorem{defn}[thm]{Definition}
\newtheorem{exam}[thm]{Example}
\newtheorem{conj}[thm]{Conjecture}
\newtheorem{quest}[thm]{Question}

\theoremstyle{remark}
\newtheorem{rem}[thm]{Remark}

\begin{center}
\vskip 1cm{\LARGE\bf The Many Faces of the Kempner Number
}
\vskip 1cm
\large
Boris~Adamczewski\\
CNRS, Universit\'e de Lyon, Universit\'e Lyon 1\\
Institut Camille Jordan  \\
43 boulevard du 11 novembre 1918 \\
69622 Villeurbanne Cedex \\
France\\
\href{mailto:Boris.Adamczewski@math.univ-lyon1.fr}{\tt Boris.Adamczewski@math.univ-lyon1.fr} \\
\end{center}

\vskip .2 in

\theoremstyle{plain}
\newtheorem*{thm2}{Theorem BBCP} 
\newtheorem*{thmRo}{Roth's Theorem} 
\newtheorem*{thmRi}{Ridout's Theorem}  
\newtheorem*{thmSub}{Subspace Theorem} 
\newtheorem*{thmCZ}{Theorem CZ} 
\newtheorem*{thmAB}{Theorem AB1} 
\newtheorem*{thmF}{Folding Lemma} 
\newtheorem*{thmSh1}{Theorem Sh1} 
\newtheorem*{thmSh2}{Theorem Sh2} 
\newtheorem*{thmAB2}{Theorem AB2} 

\newtheorem{notn}{Assumption--Notation}

\newcommand{\KK}{K\langle\langle x_1,\ldots,x_n\rangle\rangle}
\newcommand{\KKn}{K[[x_1^{\mathbb{Z}},\ldots , x_n^{\mathbb{Z}}]]}
\newcommand{\KKxy}{K\langle\langle x_1,\ldots,x_n,y_1,\ldots y_n\rangle\rangle}
\newcommand{\KKxnyn}{K[[x_1^{\mathbb{Z}},\ldots , x_{n-1}^{\mathbb{Z}};y_1^{\mathbb{Z}},\ldots , y_{n-1}^{\mathbb{Z}} ]]}

\newcommand{\KKxyy}{K[[x_1^{\mathbb{Z}},\ldots , x_n^{\mathbb{Z}};y_n^{\mathbb{Z}} ]]}
\newcommand{\KKnn}{K\langle\langle x_1,\ldots,x_{n-1}\rangle\rangle}
\newcommand{\C}{\mathcal}
\newcommand{\F}{\mathfrak}
 
\begin{abstract}  
In this survey, we present five different proofs for the transcendence of
Kempner's number, defined by the infinite series 
$\sum_{n=0}^{\infty} \frac{1}{2^{2^n}}$. 
We take the opportunity to mention some interesting ideas 
and methods that are used for proving deeper results. We outline proofs for some of these results and also point out references 
where the reader can find all the details.  
\end{abstract}

\begin{center}
{\it \`A Jean-Paul Allouche, pour son soixanti\`eme anniversaire.}
\end{center}

\begin{flushright}
   {\it Le seul v\'eritable voyage, le seul bain de Jouvence, \\
   ce ne serait pas d'aller vers de nouveaux paysages,\\ 
   mais d'avoir d'autres yeux...} 

\smallskip
   
{Marcel Proust, {\it \`A la recherche du temps perdu}}

\end{flushright}

\section{Introduction}

Proving that a given real number is transcendental is usually an extremely difficult task. 
Even for classical constants like $e$ and $\pi$, the proofs are by
no means easy, and most mathematicians would 
be happy with a single proof of the transcendence of $e+\pi$ or  $\zeta(3)$.   
In contrast, this survey will focus on the simple series
$$
\kappa := \sum_{n=0}^{\infty} \frac{1}{2^{2^n}}  
$$
that can be easily proved to be transcendental. The first proof is 
due to Kempner \cite{Kempner} in 1916 and, in honor of this result, we refer to $\kappa$ as the 
{\it Kempner number}\footnote{The number $\kappa$ is sometimes erroneously called the Fredholm number 
(see, for instance, the discussion in 
\cite{Sha}).}.  
If the transcendence of $\kappa$ is not a real issue,
our aim is instead to look at  
the many faces of $\kappa$, which will lead us to give five different proofs of this fact.  
This must be (at least for the author) some kind of record,
even if we do not claim  this list of proofs to be exhaustive. 
In particular, we will not discuss Kempner's original proof.
Beyond the transcendence of $\kappa$, the different proofs we give all
offer the opportunity  to mention some interesting ideas and methods
that are used for proving deeper results. We outline proofs for some of
these results and also point out references where the reader can find
all the details.

The outline of the paper is as follows. In Section \ref{section: knight}, we start this survey  
with a totally elementary proof of the transcendence of the Kempner 
number, based on a digital approach. Quite surprisingly, a digital approach very much 
in the same spirit has a more striking consequence concerning the problem of finding good lower bounds for 
the number of non-zero digits among the first $N$ digits of the binary expansion of algebraic irrational numbers. 
In Section \ref{section: mahler}, we give a second proof that relies
on Mahler's method. We also take the opportunity
to discuss a little-known application 
of this method to transcendence in positive characteristic. Our third proof is a consequence  of a $p$-adic version of Roth's theorem due to Ridout. It 
is given in Section \ref{section: roth}. More advanced consequences of the Thue--Siegel--Roth--Schmidt method are then outlined. 
In Section \ref{section: folding}, we give a description of the continued fraction expansion of $\kappa$  which 
turns out to have interesting consequences. We present two of them, one
concerning a question of Mahler about the Cantor set  and the other
the failure of Roth's theorem in positive characteristic. Our last two
proofs rely on such a description and the Schmidt subspace theorem. They are 
given, respectively, in Sections \ref{sec: cf1} and \ref{sec: cf2}. 
The first one uses the fact that $\kappa$ can be well approximated by a familly of 
quadratic numbers of a special type, while the second one uses
the fact that $\kappa$ and $\kappa^2$ have very good rational 
approximations with the same denominators.
Both proofs give rise to deeper results that are described briefly. 
 
Throughout this paper, $\lfloor  x\rfloor$  and $\lceil x\rceil$
denote, respectively, the floor and the ceiling of the real number
$x$.  We also use the classical notation $f(n)\ll g(n)$ (or
equivalently $g(n)\gg f(n)$), which means that  there exists a positive
real number $c$, independent of $n$, such that  $f(n)< cg(n)$
for all sufficiently large integers $n$.

\section{An ocean of zeros}\label{section: knight}

We start this survey with a totally elementary proof of the
transcendence of the Kempner number, due to Knight \cite{Kn91}. This
proof, which is based on a digital approach, is also reproduced in the
book of Allouche and Shallit \cite[Chap.\ 13]{AS}.

\begin{proof}[First proof]
Set 
$$
f(x) := \sum_{n=0}^{\infty} x^{2^n}  ,
$$
so that $\kappa = f(1/2)$. For every integer $i\geq 0$, we let $a(n,i)$ denote
the coefficient of $x^n$ in the formal power series expansion of $f(x)^i$. 
Thus $a(n,i)$ is equal to the number of ways that $n$ can be written as a sum of $i$ powers of $2$, 
where different orderings are counted as distinct. 
For instance, $a(5,3)=3$ since 
$$
5 = 1+2+2=2+1+2=2+2+1  .
$$
 Note that for positive integers 
$n$ and $i$, we clearly have 
\begin{equation}
\label{eq: ank}
a(n,i) \leq (1+\log_2 n)^{i}  .
\end{equation} 
The expression 
$$
\kappa^{i} = f(1/2)^{i} = \sum_{n=0}^{\infty} \frac{a(n,i)}{2^n}
$$
can be though of as a ``fake binary expansion" of $\kappa^{i}$ in which carries have not been yet performed.

Let us assume, to get a contradiction, that $\kappa$ is an algebraic number. 
Then there exist integers $a_0,\ldots,a_d$, with $a_d>0$, such that 
$$
a_0 + a_1 \kappa+\cdots + a_d\kappa^d  =0  .
$$
Moving all the negative coefficients to the right-hand side,
we obtain an equation of the form
\begin{equation}\label{eq: fake}
 a_{i_1}\kappa^{i_1} + \cdots + a_{i_r}\kappa^{i_r} + a_d\kappa^{d}  =  b_{j_1}\kappa^{j_1} + \cdots + b_{j_s}\kappa^{j_s}  ,
\end{equation}
where $r+s=d$, $0\leq i_1<\cdots <i_r<d$, $0\leq j_1<\cdots <j_s$, and coefficients on both sides are nonnegative. 

Let $m$ be a positive integer and set $N := (2^{d}-1)2^{m}$, so that the binary expansion of $N$ is given by
$$
(N)_2 = \underbrace{1\cdots 1}_{d}\; \underbrace{0\cdots 0}_{m}  .
$$
Then for every integer $n$ in the interval $I := [N - (2^{m-1}-1) ,  N + 2^{m}-1]$ and every integer $i$, $0\leq i\leq d$, 
we have 
$$
a(n,i) = \left\{
    \begin{array}{ll}
        d!, & \mbox{if } n=N \mbox{ and } i=d; \\
        0,& \mbox{otherwise.}
    \end{array}
\right.
$$
Indeed, every $n\not=N\in I$ has more than $d$ nonzero digits in its binary expansion, while $N$ has exactly 
$d$ nonzero digits.

Now looking at Equality (\ref{eq: fake}) as an equality between two fake binary numbers, we observe that
\begin{itemize}

\item[$\bullet$] on the right-hand side, all fake digits with position in $I$ are zero (an ocean of zeros),

\item[$\bullet$] on the left-hand side, all fake digits with position in $I$ are zero except for the one in position $N$ 
that is equal to $a_dd!$ (an island).

\end{itemize}

Note that $d$ is fixed, but we can choose $m$ as large as we want. 
Performing the carries on the left-hand side of (\ref{eq: fake}) for sufficiently large $m$, we see that the fake digit $a_dd!$ will produce 
some nonzero binary digits in a small (independent of $m$) neighborhood of the position $N$. On the other hand,  the upper bound (\ref{eq: ank})  
ensures that, for sufficiently large $m$, carries on the right-hand side of (\ref{eq: fake}) will never reach this neighborhood of the position $N$. 
By uniqueness of the binary expansion, Equality (\ref{eq: fake}) is thus impossible. This provides a contradiction. 
\end{proof}

\subsection{Beyond Knight's proof} 
Unlike $\kappa$, which is a number whose binary expansion contains  
absolute oceans of zeros, it is expected that all algebraic irrational real numbers have essentially random binary expansions 
(see the discussion in Section \ref{section: roth}).  As a consequence,  if $\xi$ is an algebraic irrational number and 
if $\mathcal P(\xi,2,N)$ denotes the number of $1$'s 
among the first $N$ digits of the binary expansion of  $\xi$, we should have  
$$
\mathcal P(\xi,2,N) \sim \frac{N}{2}  \cdot
$$
Such a result seems to be out of reach of current approaches,
and to find  good lower bounds 
for $\mathcal P(\xi,2,N)$ remains a challenging problem. 

A natural (and naive) approach to study this question can be roughly described as follows: 
if the binary expansion of $\xi$ contains too many zeros among its first digits, then some partial sums 
of its binary expansion should provide very good rational approximations to $\xi$; 
but on the other hand, we know that algebraic irrationals  
cannot be too well approximated by rationals.   
More concretely, we can argue as follows.  
Let $\xi:= \sum_{i\geq 0}1/2^{n_i}$ be 
a binary algebraic number. Then there are integers $p_k$ 
such that  
$$
\displaystyle\sum_{i=0}^k \frac{1}{2^{n_i}}= \frac{p_k}{2^{n_k}} \;\;\mbox{ and }\;\;
\left \vert  \xi - \frac{p_k}{2^{n_k}} \right\vert < \frac{2}{2^{n_{k+1}}} \cdot
$$ 
On the other hand, since $\xi$ is algebraic,  given a positive $\varepsilon$, 
Ridout's theorem (see Section \ref{section: roth})  implies that 
$$
\left\vert \xi - \frac{p_k}{2^{n_k}} \right\vert > \frac{1}{2^{(1+\varepsilon)n_k}} ,
$$
for every sufficiently large integer $k$.   
This gives that  $n_{k+1} < (1 + \varepsilon) n_k+1$ for such  $k$. Hence,  for any positive number $c$, we have 
\begin{equation}\label{ri}
\mathcal P(\xi,2,N) >  c \log N ,
\end{equation}
for every  sufficiently large $N$.

Quite surprisingly, a digital approach very much in the same spirit as Knight's  proof of the transcendence of $\kappa$ led 
Bailey, J. M. Borwein, Crandall, and Pomerance \cite{BBCP04} to obtain the following significant improvement of 
(\ref{ri}).

\begin{thm2}
 Let $\xi$ be an algebraic real number of degree $d\geq 2$. Then there exists an explicit positive number $c$ such that
\begin{equation*}
\mathcal P(\xi,2,N) > cN^{1/d}  ,
\end{equation*} 
for every  sufficiently large $N$.  
\end{thm2}

\begin{proof}[(Sketch of) proof]  We do not give all the details, for which we refer the reader to \cite{BBCP04}.  
Let $\xi$ be an algebraic number of degree $d\geq 2$, for which we assume that 
\begin{equation}
\label{eq: H}
\mathcal P(\xi,2,N) < c N^{1/d}  ,
\end{equation} 
for some positive number $c$. 
Let $a_0,\ldots,a_d$, $a_d>0$ such that 
$$
a_0+a_1\xi +\cdots +a_d\xi^d=0  .
$$
Let $\sum_{i\geq 0} 1/2^{n_i}$ denote the binary expansion of $\xi$ and set $f(x) := \sum_{i\geq 0} x^{n_i}$. We also 
let $a(n,i)$ denote the coefficient of $x^n$ in the power series expansion of $f(x)^i$.  
Without loss of generality we can assume that $n_0=0$. 
This assumption is important, in fact, for it ensures that 
$$
a(n,d-1)=0 \implies a(n,i)=0, \mbox{ for every }i, 0\leq i\leq d-1 .
$$   
Set $T_i(R):=\sum_{m\geq 1} a(R+m,i)/2^m$ and $T(R):=\sum_{i=0}^d a_iT_i(R)$. 
A fundamental remark is that $T(R)\in\mathbb Z$. 
Let $N$ be a positive integer and set $K:=\lceil 2d\log N\rceil$. Our aim is now to estimate the quantity 
$$
\sum_{R=0}^{N-K} \vert T(R)\vert  .
$$

\noindent{\it Upper bound.}  We first note that  
\begin{equation}\label{eq: ineq}
a(n,i)\leq {n+i-1\choose i-1} \mbox{ and } 
\sum_{R=0}^Na(R,i) \leq \mathcal P(\xi,2,N)^{i}  .
\end{equation}
 Using these inequalities, it is possible to show that  
\begin{eqnarray}\
\sum_{R=0}^{N-K} T_i(R) & = & \sum_{m=1}^{\infty} 2^{-m} \sum_{R=0}^{N-K} a(R+m,i)\nonumber  \\ 
&<& \sum_{R=0}^{N} a(R,i) + 2^{-K} \sum_{R=K}^N T_i(R) \nonumber \\
 & \leq & \mathcal P(\xi,2,N)^i+1  ,\nonumber 
\end{eqnarray}
for $N$ sufficiently large. We thus obtain that  
\begin{eqnarray}\label{eq: maj}
\sum_{R=0}^{N-K} \vert T(R)\vert & \leq & \sum_{i=1}^d \vert a_k\vert \left(\mathcal P(\xi,2,N)^{i}+1\right) \nonumber \\ 
&\leq &   a_dc^dN + O(N^{1-1/d})  .
\end{eqnarray}

\noindent{\it Lower bound.} We first infer from (\ref{eq: H}) and (\ref{eq: ineq}) that  
$$
\mbox{Card} \left\{ R \in [0,N] \mid a(R,d-1)>0\right\} < c^{d-1}N^{1-1/d}  .
$$
Let  $0=R_1<R_2<\cdots <R_M$ denote the elements of this set, so that $M<c^{d-1}N^{1-1/d}$. 
Set also $R_{M+1}:=N$. Then 
$$
\sum_{i=1}^M(R_{i+1}-R_i) = N .
$$
Let $\delta >0$ and set 
$$
\mathcal I := \left\{ i \in [0,M] \mid R_{i+1}-R_i \geq \frac{\delta}{3}c^{1-d}N^{1/d} \right\}  .
$$
Then we have 
\begin{equation}\label{eq: min1}
\sum_{i\in \mathcal I} \left( R_{i+1} - R_i\right)\geq \left(1-\frac{\delta}{3}\right) N  .
\end{equation}
Now let $i\in \mathcal I$. 
Note that Roth's theorem (see Section \ref{section: roth}) allows us to control the size of  blocks 
of consecutive zeros that may occur in the binary expansion of $\xi$. 
Concretely, it ensures the existence of an integer 
$$
j_i\in\left( \frac{1}{2+\delta/2}(R_{i+1}-R_i-d\log N), (R_{i+1}-R_i-d\log N)\right)
$$ 
such that $a(j_i,1)>0$. Thus $a(R_i+j_i,d)>0$ since by assumption $n_0=0$, and then 
a short computation gives that  $T(R_i+j_i-1)>0$.  

By definition, $a(R,d-1)=0$ for every $R\in (R_{i},R_{i+1})$ and thus  
$a(R,i)=0$ for every $R\in (R_i,R_{i+1})$ and every $i\in[0,d-1]$. 
For such integers $R$, a simple computation gives 
$$
T(R-1) = \frac{1}{2}T(R) +\frac{1}{2} a_da(R,d)  
$$
and thus $T(R)>0$ implies $T(R-1)>0$. 
Applying this argument successively to $R$ equal to 
$R_i+j_i-1,R_i+j_i-2,\ldots, R_i+1$, 
we finally obtain that $T(R)>0$ for every integer $R\in [R_i,R_i +j_i)$.  
The number of integers $R\in [0,N]$ such that $T(R)>0$ is thus at least equal to 
$$
\sum_{i\in\mathcal I} \frac{1}{2+\delta/2}\left( R_{i+1}-R_i-d\log N\right)  ,
$$
which, by (\ref{eq: min1}),  is at least equal to $(1/2 -\delta/3)N$ for sufficiently large $N$. 
Since $T(R)\in \mathbb Z$, we get that  
\begin{equation}
\label{eq: min2}
\sum_{R=0}^{N-K} \vert T(R)\vert \geq \left(\frac{1}{2}-\frac{\delta}{3}\right)N  ,
\end{equation}
for sufficiently large $N$.

\noindent {\it Conclusion.} For sufficiently large $N$, Inequalities (\ref{eq: maj}) and (\ref{eq: min2}) are incompatible  as soon as 
$c\leq ((2+\delta)a_d)^{-1/d}$.  Thus, choosing $\delta$ sufficiently small, this proves the theorem for any choice of 
$c$ such that $c<(2a_d)^{-1/d}$.   
\end{proof}

We end this section with a few comments on Theorem BBCP.

\begin{itemize}

\item[$\bullet$]  
It is amusing to note that replacing Roth's theorem by Ridout's theorem in this proof 
only produces a minor improvement: 
the constant $c$ can be replaced by a slightly larger one (namely by any $c< a_d^{-1/d}$).

\item[$\bullet$] A deficiency of Theorem BBCP is that it is not effective: it does not give an explicit integer $N$ above which the lower bound 
 holds. This comes from the well-known fact that Roth's theorem is itself ineffective. 
The authors of  \cite{AdFa} show that one can replace  
Roth's theorem 
by the much weaker Liouville inequality to derive an effective version of Theorem BBCP. 
This version is actually slightly weaker, because the constant $c$ is replaced by a smaller constant, but the proof becomes 
both totally elementary and effective.

\item[$\bullet$] Last but not least: Theorem BBCP immediately implies  the transcendence of the number 
$$
\sum_{n=0}^{\infty} \frac{1}{2^{\left\lfloor n^{\log\log n}\right\rfloor}} ,
$$
for which no other proof seems to be known!

\end{itemize}

\section{Functional equations}\label{section: mahler}

Our second proof of the transcendence of $\kappa$ follows a classical approach due to Mahler. 
In a series of three papers \cite{Mah29,Mah30A,Mah30B} published in 1929 and 1930, Mahler initiated a 
totally new direction in transcendence theory.   {\it Mahler's method} 
aims to prove transcendence and algebraic independence of values at algebraic points of 
locally analytic functions satisfying certain type of functional equations. 
In its original form, it concerns equations of the form   
\begin{equation*}
f(x^k) = R(x,f(x))  , 
\end{equation*}
where $R(x,y)$ denotes a bivariate rational function with coefficients in a number field.   
In our case, we consider the function 
$$
f(x) := \sum_{n=0}^{\infty} x^{2^n}  ,
$$
and we will use the fact that it is  analytic in the open unit disc and satisfies the following basic functional equation: 
\begin{equation}\label{eq: fe}
f(x^2) = f(x) - x  . 
\end{equation}

Note that we will in fact prove much more than the transcendence of $\kappa=f(1/2)$, for we will obtain the transcendence of $f(\alpha)$ 
for every nonzero algebraic number $\alpha$ in the open unit disc.  This is a typical advantage when using Mahler's method. 
Before proceeding with the proof we need to recall
a few preliminary results. 

\medskip

\noindent {\it Preliminary step 1.} The very first step of Mahler's method consists in showing that the function $f(x)$ is transcendental 
over the field of rational function $\mathbb C(x)$.  There are actually several ways to do that.  Instead of giving an elementary but {\it ad hoc} proof, 
we prefer to give the following general statement  that turns out to be useful in this area. 

\begin{thm}
Let $(a_n)_{n\geq 0}$ be an aperiodic sequence with values in a finite subset of $\mathbb Z$. 
Then $f(x)=\sum_{n\geq 0}a_nx^n$ is transcendental over $\mathbb C(x)$.   
\end{thm}

\begin{proof} Note that $f(x)\in\mathbb Z[[x]]$ has radius of convergence one and  the classical theorem of P\'olya--Carlson\footnote{Note that this 
argument could also be replaced by the use of two important results from automata theory: the Cobham and Christol theorems (see \cite{AS} and also Section \ref{allouche} for another use of Christol's theorem). } thus 
implies that $f(x)$ is either rational or transcendental. Furthermore, since the coefficients of $f(x)$ take only finitely many distinct values  
and form an aperiodic sequence, we see that $f(x)$ cannot be a rational function.  
\end{proof}

\noindent{\it Preliminary step 2.} We will also need to use Liouville's inequality as well as basic estimates about {\it height functions}.  
There are, of course, several notions of heights. The most convenient works with the absolute logarithmic 
Weil height that will be denoted by $h$.   We refer the reader 
to  the monograph of Waldschmidt \cite[Chap.\ 3]{Wa_book} for an excellent introduction to heights and 
in particular for a definition of $h$.   
Here we just recall a few  
basic properties of $h$ that will be used in the sequel. All are proved in \cite[Chap. 3]{Wa_book}.   
For every integer $n$ and every pair of algebraic numbers $\alpha$ and $\beta$, we have
\begin{equation}\label{eq: height1}
h(\alpha^n)=\vert n\vert h(\alpha) 
\end{equation}
and
\begin{equation}\label{eq: height1bis}
 h(\alpha+\beta) \leq h(\alpha)+h(\beta) +\log 2  .
 \end{equation}
More generally, if $P(X,Y)\in\mathbb Z[X,Y]\setminus\{0\}$ then
\begin{equation}\label{eq: height2}
h(P(\alpha,\beta)) \leq \log L(P) + (\deg_XP)h(\alpha) + (\deg_Y P)h(\beta)  ,
\end{equation}
where $L(P)$ denote the length of $P$, which is classically defined as the sum of the absolute values of the coefficients of $P$. 
We also recall Liouville's inequality: 
\begin{equation}\label{eq: liouville}
\log\vert \alpha\vert \geq - d h(\alpha)  ,
\end{equation}
for every nonzero algebraic number $\alpha$ of degree at most $d$.

We are now ready to give our second proof of  transcendence for $\kappa$. 

\begin{proof}[Second proof]
Given a positive integer $N$, we choose a nonzero bivariate polynomial $P_N\in \mathbb Z[X,Y]$ 
whose degree in both $X$ and $Y$ is at most $N$, 
and such that the order of vanishing at $x=0$ of the formal power series 
$$
A_N(x) := P_N(x,f(x))
$$
is at least equal to $N^2$. Note that looking for such a polynomial amounts to solving a homogeneous linear system over $\mathbb Q$ 
with $N^2$ equations and $(N+1)^2$ unknowns, which is of course always possible.  
The fact that $A_N(x)$ has a large order of vanishing at $x=0$ ensures that $A_N$ takes very small values around the origin. 
More concretely, for  every complex number $z$, $0\leq \vert z\vert <1/2$, we have 
\begin{equation}\label{eq: anmaj}
\vert A_n(z) \vert \leq c(N) \vert z\vert^{N^2}  ,
\end{equation}
for some positive $c(N)$ that only depends on $N$.

Now we pick an algebraic number $\alpha$, $0< \vert \alpha\vert <1$, and we assume that $f(\alpha)$ is also algebraic. 
Let $L$ denote a number field that contains both $\alpha$ and $f(\alpha)$ and let 
$d:=[L:\mathbb Q]$ be the degree of this extension. The functional equation (\ref{eq: fe}) implies the following  
for every positive integer $n$: 
$$
A_N(\alpha^{2^{n}}) = P_N(\alpha^{2^{n}},f(\alpha^{2^{n}})) = P_N\left(\alpha^{2^{n}},f(\alpha)- \sum_{k=0}^{n-1} \alpha^{2^k}\right) \in L  . 
$$
Thus $A_N(\alpha^{2^{n}})$ is always an algebraic number of degree at most $d$. 
Furthermore, we claim that $A_N(\alpha^{2^{n}})\not=0$ for all sufficiently large $n$. Indeed, the function $A_N(x)$ is 
analytic in the open unit disc and it is nonzero 
because $f(x)$ is transcendental over $\mathbb C(x)$, hence the identity theorem applies. 

Now, using (\ref{eq: height1}), (\ref{eq: height1bis}) and 
(\ref{eq: height2}), we obtain the following upper bound for the height of $A_N(\alpha^{2^n})$:
$$
\begin{array}{lll}
h(A_N(\alpha^{2^n})) & = &h(P_N(\alpha^{2^{n}},f(\alpha^{2^{n}})))\\  \\ 
& \leq & \log L(P_N) + Nh(\alpha^{2^n}) + Nh(f(\alpha^{2^{n}})) \\ \\
& = & \log L(P_N) + 2^nNh(\alpha) + Nh(f(\alpha) - \sum_{k=0}^{n-1} \alpha^{2^k}) \\ \\
& \leq & \log L(P_N) + 2^{n+1}Nh(\alpha) + Nh(f(\alpha)) + n\log 2  .\\ \\ 
\end{array}
$$
 From now on, we assume that $n$ is sufficiently large to ensure that 
$A_N(\alpha^{2^{n}})$ is nonzero and that $\vert \alpha^{2^n}\vert < 1/2$. 
Since $A_N(\alpha^{2^n})$ is a nonzero algebraic number of degree at most $d$, Liouville's inequality (\ref{eq: liouville}) 
implies that 
$$
\log \vert A_N(\alpha^{2^n})\vert \geq - d \left(\log L(P_N) + 2^{n+1}Nh(\alpha) + Nh(f(\alpha)) + n\log 2\right)  .
$$
On the other hand,  since $\vert \alpha^{2^n}\vert < 1/2$, Inequality (\ref{eq: anmaj}) gives that 
$$
\log  \vert A_N(\alpha^{2^n})\vert \leq \log c(N) +  2^nN^2\log \vert \alpha\vert  .
$$
We thus deduce that  
$$
 \log c(N) +  2^nN^2\log \vert \alpha\vert \geq - d \left(\log L(P_N) + 2^{n+1}Nh(\alpha) + Nh(f(\alpha)) + n\log 2\right) .
$$
Dividing both sides by $2^n$ and letting $n$ tend to infinity, we obtain
$$
N  \leq  \frac{2dh(\alpha)}{\vert \log \vert\alpha\vert \vert}  \cdot 
$$
Since $N$ can be chosen arbitrarily large independently of the choice of $\alpha$, this provides a contradiction and concludes the proof. 
\end{proof}

\subsection{Beyond Mahler's proof}\label{allouche} 

Mahler's method has, by now,
become a classical chapter in transcendence theory. 
As observed by Mahler himself, his approach allows one to deal with functions of several variables 
and systems of functional equations as well.  
It also leads to algebraic independence results, transcendence measures, measures of algebraic independence, and so forth.
Mahler's method was later developed by various authors, including Becker, Kubota, Loxton and 
van der Poorten, Masser, Nishioka, and T\"opfer, among  others. 
It is now known to apply to a variety of numbers defined by their decimal expansion, their continued fraction expansion, or as infinite products. 
For these classical aspects of Mahler's theory, we refer the reader to the monograph of Ku.\ Nishioka \cite{Ni_liv} and the references therein.

We end this section by pointing out another feature of Mahler's method that is unfortunately less well known. 
A major deficiency of Mahler's method is that, in contrast with the Siegel 
$E$- and $G$-functions,  
there is not a single classical transcendental constant that is known to be the value at an algebraic point of an analytic function 
solution to a Mahler-type functional equation. Roughly, this means that the most interesting complex numbers for number theorists 
seemingly remain beyond the scope of Mahler's method. However, 
a remarkable discovery of Denis is that Mahler's
method can be applied to prove transcendence and algebraic 
independence results involving
{\em periods of $t$-modules},
which are variants of the more classical periods
of abelian varieties, in the framework of the arithmetic of function 
fields of
positive characteristic. For a detailed discussion on
this topic, we refer the reader to the recent survey by Pellarin  \cite{Pel2}, and also \cite{Pel1}. 
Unfortunately, we cannot begin to do justice here to this interesting topic.
We must be content
to give only a hint about the proof of the transcendence of 
an analogue of $\pi$ using Mahler's method, and we hope that the interested reader will look for more in \cite{Pel1,Pel2}.  

Let $p$ be a prime number and $q=p^{e}$ be an integer power of $p$ with $e$ positive. 
We let $\mathbb F_q$ denote the finite field of $q$ elements,
$\mathbb F_q[t]$ the ring of polynomials with coefficients in $\mathbb F_q$,
and $\mathbb F_q(t)$ the field of rational functions.
We define an absolute value on $\mathbb F_q[t]$ by
$\vert P\vert = q^{\deg_t P}$ so 
that $\vert t\vert =q$. This absolute value naturally extends to $\mathbb F_q(t)$. We let $\mathbb F_q((1/t))$ denote the completion of 
$\mathbb F_q(t)$ for this absolute value and
let $C$ denote the completion of the algebraic closure of $\mathbb F_q((1/t))$ for the unique extension 
of our absolute value to the algebraic closure of $\mathbb F_q((1/t))$. 
Roughly, this allows to  replace the natural inclusions 
$$\mathbb Z\subset \mathbb Q\subset \mathbb R\subset \mathbb C$$ by the following ones   
$$F_q[t]\subset F_q(t)\subset F_q((1/t))\subset C .$$ The field $C$ is a good analogue for $\mathbb C$ and allows one to use some tools from 
complex analysis such as the identity theorem. 
In this setting, the formal power series
$$
\Pi := \prod_{n=1}^{\infty} \frac{1}{1- t^{1-q^n}} \in \mathbb F_q((1/t))\subset C
$$
can be thought of as an analogue of the number $\pi$. To be more precise, 
the Puiseux series
$$
\widetilde{\Pi} = t(-t)^{1/(q-1)} \prod_{n=1}^{\infty} \frac{1}{1- t^{1-q^n}} \in C
$$
is a fundamental period of Carlitz's module and, in this respect,  it appears to be a reasonable analogue for $2i\pi$.  
Of course, proving the transcendence of either $\Pi$ or $\widetilde{\Pi}$ over $\mathbb F_p(t)$ remains the same. 
As discovered by Denis \cite{Denis}, it is possible to deform the infinite product given in our definition of  $\Pi$,  in order to obtain the following  
``analytic function" 
$$
f_{\Pi}(x) :=  \prod_{n=1}^{\infty} \frac{1}{1- tx^{q^n}} 
$$
which converges for all $x\in C$ such that $\vert x\vert <1$. 
A remarkable property is that the function $f_{\Pi}(x)$ satisfies the following Mahler-type functional equation:
$$
f_{\Pi}(x^q) = \frac{f_{\Pi}(x)}{(1-tx^q)}  \cdot
$$
As the principle of Mahler's method also applies in this framework, one can prove along the same lines as in the proof we just gave 
for the transcendence of $\kappa$ that  $f_{\Pi}$ takes transcendental values at every nonzero algebraic point 
in the open unit disc of $C$. Considering the rational point $1/t$,  
we obtain the transcendence of $\Pi = f_{\Pi}(1/t)$. 

Note that there are many other proofs of the transcendence of $\Pi$. The first is due to Wade \cite{Wa} in 1941.  
Other proofs were then given by Yu  \cite{Yu} using the theory of Drinfeld
modules,  by Allouche \cite{All90} using automata theory and 
Christol's theorem, and by De Mathan \cite{DM}  using tools from Diophantine
approximation.

\section{$p$-adic rational approximation}\label{section: roth}

The first transcendence proof that graduate students in mathematics usually meet 
concerns the so--called Liouville number
$$
{\mathcal L}:= \sum_{n=1}^{\infty} \frac{1}{b^{n!}}  \cdot
$$
This series is converging so quickly that partial sums 
$$
\frac{p_n}{q_n} := \sum_{k=1}^{n}\frac{1}{b^{k!}}
$$ 
provide infinitely many extremely good rational approximations to $\mathcal L$, namely 
$$
\left\vert {\mathcal L} - \frac{p_n}{q_n} \right\vert   
 < \frac{2}{q_n^{n+1}}  \cdot
$$
In view of the classical Liouville inequality \cite{Li}, these approximations prevent $\mathcal L$ from being algebraic.  
Since  $\kappa$ is also defined by a lacunary series that converges very fast, it is tempting to try 
to use a similar approach. However, we will see that this requires much more sophisticated tools.

Liouville's inequality is actually enough to prove the transcendence for series such as
$\sum_{i=0}^{\infty} 1/2^{n_i}$,  
where $\limsup (n_{i+1}/n_i)=+\infty$, but it does not apply if $n_i$ has only an exponential growth like 
$n_i=2^{i}$, $n_i=3^{i}$ or $n_i=F_i$ (the $i$th Fibonacci number).  In the case where $\limsup (n_{i+1}/n_i)>2$, 
we can use Roth's theorem \cite{Ro55}.

\begin{thmRo}  Let $\xi$ be a real algebraic number and $\varepsilon$ 
be a positive real number. Then the inequality 
$$
\left\vert \xi - \frac{p}{ q} \right\vert <  \frac{1}{q^{2 + \varepsilon}} 
$$
has only a finite number of rational solutions $p/q$. 
\end{thmRo}

For instance, the transcendence of  the real number 
$\xi := \sum_{i=0}^{\infty} 1/2^{3^{i}}$
 is now a direct consequence of the inequality
$$
0<\left\vert \xi - \frac{p_n}{q_n} \right\vert   
 < \frac{2}{q_n^3}  ,
 $$
where $p_n/q_n := \sum_{i=0}^n 1/2^{3^{i}}$.  
However, the same trick does not apply to $\kappa$, for we  get that 
$$
\left\vert \kappa - \frac{p_n}{q_n} \right\vert   
 \gg \frac{1}{q_n^2}  ,
 $$
if $p_n/q_n := \sum_{i=0}^n 1/2^{2^{i}}$.

The transcendence of $\kappa$ actually requires  the following $p$-adic 
extension of Roth's theorem due to Ridout \cite{Ri57}.  
For every prime number $\ell$, 
we let $\vert \cdot \vert_\ell$ denote the $\ell$-adic absolute value 
normalized such that $\vert \ell \vert_\ell = \ell^{-1}$.

\begin{thmRi}
Let $\xi$ be an algebraic number and 
$\varepsilon$ be a positive real number. 
Let $S$ be a finite set of distinct prime numbers. 
Then the inequality 
$$
\left( \prod_{\ell \in S} \vert p \vert_\ell \cdot \vert q\vert_\ell \right) \cdot 
\left\vert \xi - \frac{p}{q} \right\vert < \frac{1}{q^{2+\varepsilon}}
$$
has only a finite number of rational solutions $p/q$. 
\end{thmRi}

With Ridout's theorem in hand, the transcendence of $\kappa$ can be easily deduced: we just have to take into account 
that the denominators of our rational approximations are powers of $2$.  

\begin{proof}[Third proof]
Let $n$ be a positive integer and 
set 
$$
\rho_n := \sum_{i=1}^{n}\frac{1}{2^{2^{i}}}.   
$$  
Then there exists an integer $p_n$ such that $\rho_n = p_n/q_n$ with 
$q_n = 2^{2^n}$.   Observe that
$$
\left\vert \kappa - \frac{p_n}{q_n} \right\vert  < \frac{2}{2^{2^{n+1}}} 
= \frac{2}{(q_n)^{2}},
$$
and let $S=\{2\}$.  
Then, an easy computation gives that 
 $$
 \vert q_n \vert_2  \cdot \vert p_n\vert_2  \cdot 
\left\vert \kappa - \frac{p_n}{q_n}\right\vert < \frac{2}{(q_n)^{3}}  \cdot
$$
Applying Ridout's theorem, we get that $\kappa$ 
is transcendental.  
\end{proof}

Of course there is no mystery, the difficulty in this proof is hidden in the proof of Ridout's theorem.

\subsection{Beyond Roth's theorem}

The Schmidt subspace theorem \cite{Schmidt80} provides a formidable multidimensional generalization of Roth's theorem. 
We state below a simplified version of the $p$-adic subspace theorem due to Schlickewei \cite{Sch77}, which  
turns out to be very useful for proving transcendence of numbers 
defined by their base-$b$ expansion or by their continued fraction expansion.  
Note that our  last two proofs of the transcendence of $\kappa$, given in Sections \ref{sec: cf1} and \ref{sec: cf2}, 
both rely on the subspace theorem.  
Several recent applications of this theorem can also be found in \cite{Bilu}.  

We recall that a {\it linear form} (in $m$ variables) is a homogeneous polynomial (in $m$ variables) of degree $1$.

\begin{thmSub}
Let $m\ge 2$ be an integer and $\varepsilon$ be a positive real number. 
Let $S$ be a finite set of distinct prime numbers. 
Let $L_1, \ldots , L_m$ 
be $m$ linearly independent linear forms in $m$ variables with real algebraic coefficients.
Then the set of solutions ${\bf x} = (x_1, \ldots, x_m)$
in $\mathbb Z^m$ to the inequality
$$
\left(\prod_{i=1}^m\prod_{\ell \in S} \vert x_i\vert_\ell \right) \cdot  \prod_{i=1}^m 
{\vert L_{i} ({\bf x}) \vert } \leq  
(\max\{|x_1|, \ldots , |x_m|\})^{-\varepsilon}
$$
lies in finitely many proper subspaces of $\mathbb Q^m$.
\end{thmSub}

Let us first see how the subspace theorem implies Roth's theorem. 
Let $\xi$ be a real algebraic number and 
$\varepsilon$ be a positive real number. Consider the two independent linear 
forms $\xi X - Y$ and $X$. 
The subspace theorem  implies that all the integer solutions 
$(p, q)$ to
\begin{equation}\label{Ch7:equation:rothlin}
\vert q \vert \cdot \vert q \xi - p\vert < \vert q\vert^{-\varepsilon} 
\end{equation}
are contained in a finite union of proper subspaces of $\mathbb Q^2$. 
There thus is a finite set of lines $x_1 X + y_1 Y = 0,\ \ldots ,\ 
x_t X + y_t Y = 0$ such that,
every solution $(p, q)\in\mathbb Z^2$ to (\ref{Ch7:equation:rothlin}), belongs to one of these lines. This means that the set of 
rational solutions $p/q$ to $\left\vert \xi - p/q \right\vert < q^{-2-\varepsilon}$ is finite, which is Roth's theorem.

\subsubsection{A theorem of Corvaja and Zannier} 
Let us return to the transcendence of $\kappa$. Given an integer $b\geq 2$ and letting $S$ denote the set of prime divisors of $b$,  
it is clear that the same proof also gives  the transcendence of $\sum_{n=0}^{\infty}1/b^{2^n}$. 
However, if we try to replace $b$ by a rational or an algebraic number, we may encounter new difficulties. As a good exercise, 
the reader can convince himself that the proof 
will still work with $b={5 \over 2}$ or $b={{17} \over 4}$,
but not with $b={3 \over 2}$ or $b={5 \over 4}$.   
Corvaja and Zannier  \cite{CZ} make clever use of the subspace theorem 
that allows them to overcome the problem in all cases. Among other results, they proved 
the following nice theorem.

\begin{thmCZ}
Let $(n_i)_{i\geq 0}$ be a sequence of positive integers such that 
$\liminf  n_{i+1}/n_i>1$  
and let $\alpha$, $0<\vert \alpha\vert<1$, be an algebraic number. 
Then the number 
$$
\sum_{i=0}^{\infty} \alpha^{n_i}
$$
is transcendental.
\end{thmCZ}

Of course,  we recover the fact, already proved in Section \ref{section: mahler} by Mahler's method, that  
the function $f(x)= \sum_{n=0}^{\infty} x^{2^n}$ takes transcendental values at every nonzero algebraic point in 
the open unit disc. The proof of Theorem CZ actually requires an extension of the $p$-adic subspace theorem to number fields 
(the version we gave is sufficient for rational points). 
We also take the opportunity to mention that the main result of \cite{Ad04} is actually a consequence of Theorem 4 
in \cite{CZ}.

In order to explain the idea of Corvaja and Zannier we somewhat oversimplify the situation by considering only 
the example of $f({4 \over 5})$.  We refer the reader to \cite{CZ} for a complete proof. 
We assume that $f({4 \over 5})$ is algebraic and we 
aim at deriving a contradiction.  
A simple computation gives that 
$$
\left\vert f\left({4 \over 5}\right) - \sum_{k=0}^n \left(\frac{4}{5}\right)^{2^k} \right\vert < 2 \left(\frac{4}{5}\right)^{2^{n+1}}  ,
$$ 
for every nonnegative integer $n$.
This inequality can obviously be rephrased as 
$$
\left\vert f\left({4 \over 5}\right) - \sum_{k=0}^n \left(\frac{4}{5}\right)^{2^k} -  \left(\frac{4}{5}\right)^{2^{n+1}} - \left(\frac{4}{5}\right)^{2^{n+2} }\right\vert 
< 2 \left(\frac{4}{5}\right)^{2^{n+3}}  ,
$$
but the subspace theorem will now take care of the fact that  the last two terms on the left-hand side are $S$-units (for $S=\{2,5\}$). 
Multiplying by $5^{2^{n+2}}$, we obtain that 
$$
\left\vert 5^{2^{n+2}}f\left({4 \over 5}\right) - 5^{2^{n+2}-2^n}p_n - 4^{2^{n+1}}5^{2^{n+1}} - 4^{2^{n+2}}\right\vert < 2 \left(\frac{4}{5}\right)^{2^{n+3}}5^{2^{n+2}} ,
$$
for some integer $p_n$. 
Consider the following four linearly independent linear forms
with real algebraic coefficients:
$$
\begin{array}{ll}
L_1 (X_1,X_2,X_3,X_4) = & f({4 \over 5})X_1-X_2-X_3-X_4,  \\
L_2 (X_1,X_2,X_3,X_4) = & X_1, \\
L_3 (X_1,X_2,X_3,X_4) = & X_3 , \\
L_4(X_1,X_2,X_3,X_4)=& X_4. 
\end{array}
$$
For every integer $n \geq 1$, consider the integer quadruple
$$
{\bf x}_n = (x_1^{(n)},x_2^{(n)},x_3^{(n)},x_4^{(n)}):=\left( 5^{2^{n+2}} , 5^{2^{n+2}-2^n}p_n ,  4^{2^{n+1}}5^{2^{n+1}} , 4^{2^{n+2}} \right)  .
$$
Note that $\Vert {\bf x}_n\Vert_{\infty}\leq 5\cdot 5^{2^{n+2}}$. 
Set also $S=\{2,5\}$. Then a simple computation shows that 
$$
\left(\prod_{i=1}^4\prod_{\ell \in S} \vert x_i^{(n)}\vert_\ell \right) \cdot  \prod_{i=1}^4 
{\vert L_{i} ({\bf x}_n) \vert } \leq  2 \left(\frac{4^8}{5^7}\right)^{2^{n}}  < 
 \Vert {\bf x}_n\Vert_{\infty}^{-\varepsilon}  ,
$$
for some $\varepsilon>0$. 
We then infer from the subspace theorem 
 that all points ${\bf x}_n$ lie in a finite number of proper subspaces of $\mathbb Q^4$. 
Thus, there exist a nonzero integer quadruple $(x,y,z,t)$ and
an infinite set of distinct positive integers ${\mathcal N}$ such that
\begin{equation}\label{eq: plan} 
 5^{2^{n+2}} x +  5^{2^{n+2}-2^n}p_n y + 4^{2^{n+1}}5^{2^{n+1}} z + 4^{2^{n+2}} t = 0,  
\end{equation}
for every $n$ in ${\mathcal N}$.
Dividing (\ref{eq: plan}) by $ 5^{2^{n+2}}$ and letting  
$n$ tend to infinity along ${\mathcal N}$, we get that 
$$
x + \kappa y=0  .
$$
Since $\kappa$ is clearly irrational, this implies that $x=y=0$.  But then Equality (\ref{eq: plan})  becomes 
$$
4^{2^{n+1}}5^{2^{n+1}} z = -4^{2^{n+2}} t  ,
$$
which is impossible for large $n\in\mathcal N$ unless $z=t=0$  (look at, for instance, the $5$-adic absolute value). 
This proves that $x=y=z=t=0$, a contradiction.

Note that  the proof of the transcendence of $f(\alpha)$, for every algebraic number $\alpha$ with $0<\vert \alpha\vert <1$,  
actually requires the use of the subspace theorem with an arbitrary large number of variables (depending on $\alpha$). 
For instance, we need $14$ variables to prove the transcendence of  
 $f(2012/2013)$.


\subsubsection{The decimal expansion of algebraic numbers} 
The decimal expansion of real numbers such as $\sqrt 2$, $\pi$, and $e$ 
appears to be quite mysterious and, for a long time, has baffled  
mathematicians.  After the pioneering work of \'E. Borel \cite{Borel1,Borel2}, 
most mathematicians expect that all  
algebraic irrational numbers are {\it normal numbers},  even if this conjecture currently seems to be out of reach.       
Recall that a real number is normal if 
for every integer $b\geq 2$ and every positive integer $n$,
each one of the $b^n$ blocks of digits of length $n$ 
occurs in its base-$b$ expansion with the same frequency.  
We end this section by pointing out an application of the $p$-adic subspace theorem related to 
this problem.

Let $\xi$ be a real number and $b\geq 2$ be a positive integer. Let $(a_n)_{n\geq -k}$ denote 
the base-$b$ expansion of $\xi$, that is,  
$$
\xi= \displaystyle\sum_{n\geq -k} \frac{a_n}{b^n} = a_{-k}\cdots a_{-1}a_0{\scriptscriptstyle \bullet} a_1a_2\cdots  .
$$
Following Morse and Hedlund \cite{HM}, we define the 
{\it complexity function} of $\xi$ with respect to the base $b$ as 
the function that associates  with each positive integer $n$ the positive integer  
$$
p(\xi,b,n) := \mbox{Card} \{(a_j,a_{j+1},\ldots, a_{j+n-1}), \; j \geq 1\}.
$$
A normal number thus has the maximum possible complexity in every integer base, 
that is, $p(\xi,b,n) = b^n$ for every 
positive integer $n$ and every integer $b\geq 2$. One usually  expects 
such a high complexity for numbers like $\sqrt 2$, $\pi$, and $e$.   Ferenczi and Mauduit \cite{FM} gave the first lower bound 
for the complexity of  all algebraic irrational numbers by means of Ridout's theorem.  
More recently, Adamczewski and Bugeaud \cite{AdBuAnnals} 
use the subspace theorem to obtain the following significant
improvement of their result.

\begin{thmAB}
 Let $b \geq 2$ be an integer and $\xi$ be an algebraic irrational number.  
Then 
\begin{equation*}\label{bfm}
\lim_{n\to \infty}  \frac{p(\xi,b,n)}{ n} = + \infty  .
\end{equation*}
\end{thmAB}

Note that Adamczewski \cite{Ad10} obtains a weaker lower bound  for some transcendental 
numbers involving the exponential function.  
For a more complete discussion concerning the complexity of the base-$b$ expansion of algebraic numbers, 
we refer the reader to \cite{Ad10,AdBuAnnals,AdBuCant,Miw09}. 

\begin{proof}[Hint of proof]  We only outline the main idea for proving 
Theorem AB1 and refer the reader 
to \cite{AdBuAnnals} or \cite{AdBuCant} 
for more details.  Let $\xi$ be an algebraic number and let us assume that 
\begin{equation}\label{eq: comp}
\liminf_{n\to \infty}  \frac{p(\xi,b,n)}{ n} < +\infty,
\end{equation}
for some integer $b\geq 2$. Our goal is thus to prove that $\xi$ is rational. Without loss of generality, we can assume that $0<\xi<1$.

Our assumption implies that the number of distinct blocks of digits of length $n$ in the base-$b$ expansion of $\xi$ is quite small 
(at least for infinitely many integers $n$).  Thus, at least some of these blocks of digits have to reoccur frequently, which forces the 
early occurrence  
of some repetitive patterns in the base-$b$ expansion of $\xi$. This rough idea can be formalized as follows. 
We first recall some notation from combinatorics on words. 
Let $V=v_1\cdots v_r$ be a finite word.   We let $\vert V\vert=r$ denote the 
length of  $V$. 
For any positive integer $k$, we write
$V^k$ for the word 
$$
\overbrace{V\cdots V}_{\mbox{$k$ times}} .
$$ 
More generally, for any positive real number
$w$,  $V^w$ denotes the word
$V^{\lfloor w \rfloor}V'$, where $V'$ is the prefix of
$V$ of length $\left\lceil(w-\lfloor w\rfloor)\vert V\vert\right\rceil$. 
With this notation, one can show that the assumption (\ref{eq: comp}) ensures the existence of a real number $w>1$ 
and of two infinite sequences of finite words 
$(U_n)_{n\geq 1}$ and $(V_n)_{n\geq 1}$ such that 
the base-$b$ expansion of $\xi$ begins with the block of digits $0{\scriptscriptstyle \bullet} U_nV_n^w$  for every positive integer $n$.  
Furthermore, if we set $r_n := \vert U_n \vert$ and $s_n := \vert V_n\vert$, we have that $s_n$ tends to infinity with $n$ and 
there exists a positive number $c$ such that 
$r_n/s_n <c$ for every $n\geq 1$.

This combinatorial property has the following Diophantine translation. 
For every positive integer $n \ge 1$,  
$\xi$ has to be close to the rational number with ultimately periodic base-$b$ expansion 
$$
0{\scriptscriptstyle \bullet} U_nV_nV_nV_n\cdots   .
$$
Precisely, one can show the existence of an  
integer $p_n$  such that
\begin{equation*}\label{Ch7:equation:pj}
\left\vert \xi - \frac{p_n}{ b^{r_n} (b^{s_n} - 1)} \right\vert
\ll  \frac{1 }{b^{r_n + w s_n}}  \cdot
\end{equation*}

Consider the following three linearly independent linear forms
with real algebraic coefficients:
$$
\begin{array}{ll}
L_1 (X_1, X_2, X_3) = & \xi X_1 - \xi X_2 - X_3,  \\
L_2 (X_1, X_2, X_3) = & X_1, \\
L_3 (X_1, X_2, X_3) = & X_2. 
\end{array}
$$
Evaluating them at the integer points 
$${\bf x}_n =(x_1^{(n)},x_2^{(n)},x_3^{(n)}):= (b^{r_n+s_n}, b^{r_n}, p_n)  ,
$$ 
we easily  obtain that 
$$
\left(\prod_{i=1}^3\prod_{p \in S} \vert x_i^{(n)}\vert_p \right) \cdot  
\prod_{i=1}^3 {\vert L_i ({\bf x}_n) \vert } \ll 
\left( \max\{b^{r_n + s_n}, b^{r_n}, p_n\} \right)^{-\varepsilon},
$$
where $\varepsilon := (\omega-1)/2(c+1) >0$ and $S$ denotes the set of
prime divisors of $b$.
We then infer from the subspace 
theorem  that all points ${\bf x}_n$
belong to a finite number of proper subspaces of $\mathbb Q^3$. 
There thus exist a nonzero integer triple $(x,y,z)$ and
an infinite set of distinct positive integers ${\mathcal N}$ such that
\begin{equation}\label{Ch7:equation:plan} 
x b^{r_n+s_n} + y b^{r_n} + z p_n  = 0 ,  
\end{equation}
for every $n$ in ${\mathcal N}$. 
Dividing (\ref{Ch7:equation:plan}) by $b^{r_n+s_n}$ and letting  
$n$ tend to infinity along ${\mathcal N}$, we get that 
$$
x+\xi z = 0  ,
$$
as $s_n$ tends to infinity. 
Since $(x,y,z)$ is a nonzero vector,  this implies that $\xi$ is a rational 
number. This ends  the proof.
\end{proof}
\section{Interlude: from base-$b$ expansions to continued fractions}\label{section: folding}

It is usually very difficult to extract any information about the continued fraction expansion of a given irrational real number 
from its decimal or binary expansion and  vice versa. For instance, $\sqrt 2$, $e$, and $\tan 1$ all have 
a very simple continued fraction expansion, while they are expected to be normal and thus  should  have essentially 
random expansions in all integer bases.  
In this section, we shall give an exception to this rule:  
our favorite binary number $\kappa$ has a predictable continued fraction expansion that enjoys remarkable 
properties involving both repetitive and symmetric patterns (see Theorem Sh1 below). Our last two proofs of transcendence for $\kappa$, 
given in Sections \ref{sec: cf1} and \ref{sec: cf2}, both rely on Theorem Sh1.

For an introduction to continued fractions, the reader is referred to standard books such as 
Perron\cite{Perron}, Khintchine \cite{Kh},
or Hardy and Wright \cite{HaWr}. 
We will use the classical notation for finite or infinite continued fractions
$$
\frac{p}{q}
= a_0 + \cfrac{1}{a_1+\cfrac{1}{a_2+\cfrac{1}{\ddots+\cfrac{1}{a_n}}}}
= [a_0, a_1,\cdots, a_n]
$$
resp., 
$$
\xi =
a_0 + \cfrac{1}{a_1+\cfrac{1}{a_2+\cfrac{1}{\ddots+\cfrac{1}
{a_n + \cfrac{1}{\ddots}}}}}  = [a_0, a_1,\cdots, a_n, \cdots]
$$
where $p/q$ is a positive rational number (resp.~$\xi$ is a positive irrational real
number), $n$ is a nonnegative integer, $a_0$ is a nonnegative integer, and the $a_i$'s 
are positive integers for $i \geq 1$. Note that we allow $a_n=1$ in the first equality.  
If $A=a_1a_2\cdots$ denotes a finite or an infinite word whose letters $a_i$ are positive integers, then the expression 
$[0,A]$ stands for the finite or infinite continued fraction $[0,a_1,a_2,\ldots]$. Also, if $A=a_1a_2\cdots a_n$ is a finite word, we 
let $A^R:=a_na_{n-1}\cdots a_1$ denote the reversal of $A$. As in the previous section, we use $\vert A\vert$ 
to denote the length of the finite word $A$.

The following elementary result was first discovered by Mend\`es France \cite{MF}\footnote{The folding lemma is an avatar of the so--called 
mirror formula, another very useful elementary identity for continued fractions, which is the object of the survey \cite{AdAl}. 
Many references to work related
to these two identities can be found in \cite{AdAl}.}.

\begin{thmF}
 Let $c, a_0, a_1, \ldots,a_n$ be positive integers.
Let $p_n/q_n := [a_0, a_1, \cdots, a_n]$. Then
\begin{equation}\label{fold}
\frac{p_n}{q_n} + \frac{(-1)^n}{c q_n^2}
= [a_0, a_1, a_2, \cdots, a_n, c, -a_n, -a_{n-1}, \cdots, -a_1].
\end{equation}
\end{thmF}

For a proof of the folding lemma, see, for instance, \cite[p.\ 183]{AS}. 
In Equality~(\ref{fold}) negative partial quotients occur. However, we have two simple rules 
that permit to get rid of these forbidden partial quotients:
\begin{equation}\label{eq: rule1}
[\ldots,a,0,b,\ldots] = [\ldots,a+b,\ldots]
\end{equation}
and 
\begin{equation}\label{eq: rule2}
[\ldots,a, -b_1,  \cdots, -b_r] = [,\ldots,a-1,1,b_1-1,b_2,\ldots,b_r]  .
\end{equation}

As first discovered independently by Shallit \cite{Sh1,Sh2} and Kmo\v{s}ek \cite{Km}, 
the folding lemma can be used to describe 
the continued fraction expansion 
of some numbers having a lacunary expansion in an integer base, such as $\kappa$.  
Following  Theorem 11 in \cite{Sh1},  we give now a complete description of the continued fraction expansion of $2\kappa$. 
The choice of $2\kappa$ instead of $\kappa$ is justified by obtaining  a nicer formula.

\begin{thmSh1}
Let $A_1:=1112111111$ and $B_1:=11121111$. For every positive integer $n$, 
let us define the finite words 
$A_{n+1}$ and $B_{n+1}$ as follows: 
$$
A_{n+1} = A_n12 (B_n)^R 
$$
and 
$$
B_{n+1} \mbox{ is  the prefix of } A_{n+1} \mbox{ with length }\vert A_{n+1}\vert -2  .
$$  
Then the sequence of words 
$A_n$ converges to an infinite word 
$$
A_{\infty} = 111 2 111111 12 1111 2 111 12 \cdots
$$
and 
 $$
 2\kappa = [1,A_{\infty}] =[1,1,1,1,2,1,1,1,1,1,1,1,2,1,1,1,1,2,1,1,1,1,2,\ldots]  .
 $$
 \end{thmSh1}

In particular, the partial quotients of $2\kappa$ take only the values $1$ and $2$. This shows that $\kappa$ is badly 
approximable by rationals and {\it a fortiori} that the transcendence of $\kappa$  is beyond the scope of Roth's theorem 
(the $p$-adic version of Roth's theorem 
was thus really needed in Section \ref{section: roth}).

\begin{proof} 
First note that by the
definition of $A_{n+1}$, the word $A_n$ is a prefix of $A_{n+1}$, for every nonnegative integer $n$, 
which implies that the sequence of finite words converges (for the usual topology on words)  to an infinite word $A_{\infty}$.

For every integer $n\geq 0$, we set 
$$
\frac{P_n}{Q_n} := \sum_{k=0}^{n} \frac{2}{2^{2^k}}  \cdot
$$
We argue by induction to prove that 
\begin{equation*}
P_n/Q_n = [1,A_{n-2}]  ,
\end{equation*}
for every integer $n\geq 3$. We first note that $P_3/Q_3=[1,1,1,1,2,1,1,1,1,1,1]=[1,A_1]$. 
Let $n\geq 3$ be an integer and let us assume that $P_n/Q_n=[1,A_{n-2}]$. 
By the definition of $P_n/Q_n$, we have that 
\begin{equation}\label{eq: pnqn}
\frac{P_{n+1}}{Q_{n+1}} = \frac{P_n}{Q_n} + \frac{1}{2Q_n^2}  \cdot
\end{equation}
Furthermore, an easy induction shows that for every integer $k\geq 1$, $\vert A_k\vert$ is even and $A_k$ ends 
with $11$ so that 
\begin{equation}\label{eq: 11}
(A_k)^R = 11 (B_k)^R  .
\end{equation}
Since $\vert A_{n-2}\vert$ is even, we can apply the folding Lemma and we infer from Equalities (\ref{eq: pnqn}) and (\ref{eq: 11}), and from 
the transformation rules (\ref{eq: rule1}) and (\ref{eq: rule2}) that 
$$
\begin{array}{lll}
P_{n+1}/Q_{n+1} &=& [1,A_{n-2},2,- (A_{n-2})^R] \\ \\
 &=& [1,A_{n-2},1,1, 0,1,(B_{n-2})^R]\\ \\
&=& [1,A_{n-2},1,2, (B_{n-2})^R] \\ \\
&=& [1,A_{n-1}]  .
\end{array}
$$
This proves that $P_n/Q_n=[1,A_{n-2}]$ for every $n\geq 3$. 
Since the sequence $(P_n/Q_n)_{n\geq 0}$ converges to $2\kappa$ and $A_n$ is always a prefix of $A_{n+1}$, 
we obtain that $2\kappa=[1,A_{\infty}]$, as desired. 
\end{proof}

\subsection{Two applications}  
In the second part of the paper \cite{Sh1}, 
Shallit \cite{Sh2} extends his construction and obtained the following general result\footnote{K\"ohler \cite{Ko} also obtains 
independently almost the same result  
after he studied  \cite{Sh1}. It is also worth mentioning that this result was somewhat anticipated, although written in a rather different form, 
by Scott and Wall in 1940 \cite{SW}.}.

\begin{thmSh2}
 Let $b\geq 2$ and $n_0\geq 0$ be integers and let 
$(c_n)_{n\geq 0}$ be a sequence of positive integers such that  $c_{n+1}\geq 2c_n$, 
for every integer $n\geq n_0$. Set $d_n:=c_{n+1}-2c_n$ and 
$$
S_b(n) := \sum_{k=0}^n \frac{1}{b^{c_k}}  \cdot
$$
If $n\geq n_0$ and $S_b(n)=[a_0,a_1,\ldots,a_r]$, with $r$ even, then 
 $$
S_b(n+1) = [a_0,a_1,\ldots,a_r,b^{d_n}-1,1,a_r-1,a_{r-1},\ldots,a_1]  .
 $$
 \end{thmSh2}

This result turns out to have interesting consequences, two of which are recalled below.

\subsubsection{A question of Mahler about the Cantor set}  Mahler \cite{Mah84} asked the following question: 
how close can irrational numbers in the Cantor set be approximated by rational numbers?
We recall that the {\it irrationality exponent}  of an irrational real number $\xi$, denoted by $\mu(\xi)$, is defined as 
the supremum of the real numbers $\mu$ for which the inequality 
$$
\left\vert \xi - \frac{p}{ q} \right\vert < \frac{1}{q^{\mu}} 
$$
has infinitely many rational solutions $p/q$. Mahler's question may thus  be interpreted as follows:  are there elements in the Cantor set 
with any prescribed irrationality exponent?

This question was first answered positively by Levesley, Salp and Velani \cite{LSV} by means of tools from metric number theory. 
A direct consequence of  Shallit's result is that one can also simply answer Mahler's question 
by providing explicit example of 
numbers in the Cantor set with any prescribed irrationality exponent. 
We briefly outline how to prove this result 
and refer the reader to \cite{Bu08} for more details. Some refinements along the same lines can also be found in \cite{Bu08}.    

Let $\tau\geq 2$ be a real number. 
Note first that the number 
$$
\xi_{\tau} := 2 \sum_{n= 1}^{\infty} \frac{1}{3^{\lfloor \tau^n\rfloor}}
$$
clearly belongs to the Cantor set. Furthermore, the partial sums of $\xi_{\tau}$ provide infinitely many good rational approximations 
which ensure that $\mu(\xi_{\tau})\geq \tau$.   
When $\tau\geq (3+\sqrt 5)/2$, a classical approach based on triangles inequalities allows to 
show that $\mu(\xi_{\tau})\leq \tau$. However, the method fails when $\tau$ satisfies $2\leq \tau<(3+\sqrt 5)/2$.  

In order to overcome this difficulty, we can use repeatedly Theorem Sh2 with $b=3$ and $c_n=\lfloor \tau^n\rfloor$ to obtain   
the continued fraction expansion of $\xi_{\tau}/2$.  Set $\xi_{\tau}/2:=[0,b_1,b_2,\ldots]$ and let  $s_n$ denote the denominator 
of the $n$th convergent to $\xi_{\tau}/2$.  If $\tau=2$, we see  that the partial quotients $b_n$ are bounded, which implies 
$\mu(\xi_2)=\mu(\xi_2/2)=2$, as desired. We can thus assume that $\tau >2$.   
Let us recall that once we know the continued fraction of an irrational number $\xi$, it becomes easy to deduce 
its irrationality exponent.  Indeed, if $\xi=[a_0,a_1,\ldots]$,  it is well-known that  
\begin{equation}\label{eq: mes}
\mu(\xi) =  2 + \limsup_{n\to\infty} \frac{\ln a_{n+1}}{\ln q_n}  ,
\end{equation}
where $p_n/q_n$ denotes the $n$th convergent to $\xi$.   Equality (\ref{eq: mes}) is actually a direct consequence of the inequality
$$
 \frac{1}{(2+a_{n+1})q_n^2}<\left\vert \xi - \frac{p_n}{q_n} \right\vert < \frac{1}{a_{n+1}q_n^2 }
$$
and the fact that the convergents provide the best rational
approximations (see, for instance, \cite[Chapter 6]{Kh}).
When $\tau >2$, the formula given in Theorem Sh2 shows that the large partial quotients\footnote{That is those which are larger than all previous ones.} 
of $\xi_{\tau}/2$ are precisely those equal to $3^{d_n}-1$ which occur first   
at some positions, say $r_n+1$. But then Theorem Sh2 implies  
that $s_{r_n}$ is the denominator of $\sum_{k=1}^n1/3^{\lfloor \tau^k\rfloor}$, that is  $s_{r_n}= 3^{\lfloor \tau^n\rfloor}$. 
A simple computation thus shows that  
$$
\limsup_{n\to\infty} \frac{\ln b_{n+1}}{\ln s_n} = \limsup_{n\to\infty} \frac{\ln (3^{d_n}-1)}{\ln 3^{\lfloor \tau^n\rfloor}}= \tau-2 ,
$$
since $d_n=\lfloor 3^{\tau^{n+1}}\rfloor - 2 \lfloor 3^{\tau^n}\rfloor$. 
Then we infer from Equality (\ref{eq: mes})  that  $\mu(\xi_{\tau})= \mu(\xi_{\tau}/2)= \tau$, as desired.

\subsubsection{The failure of Roth's theorem in positive characteristic} 
We consider now Diophantine approximation in positive characteristic. Let $\mathbb F_p((1/t))$ be the field of Laurent series 
with coefficients in the finite field $\mathbb F_p$, endowed with the natural absolute value $\vert \cdot \vert$ defined at the 
end of Section \ref{section: mahler}. 
In this setting, the approximation of real numbers by rationals is naturally replaced by the approximation of  Laurent series by rational functions.  
In analogy with the real case,  we define the irrationality exponent of $f(t) \in\mathbb F_p((1/t))$, 
denoted by $\mu(f)$, as  
the supremum of the real number $\mu$ for which the inequality 
$$
\left\vert f(t) - \frac{P(t)}{Q(t)} \right\vert < \frac{1}{\deg Q^{\mu}} 
$$
has infinitely many rational solutions $P(t)/Q(t)$.

It is well-known that Roth's theorem fails in this framework. 
Indeed, Mahler \cite{Mah49} remarked that it is even not possible to improve Liouville's bound 
for the power series 
$$
f(t): = \sum_{n=0}^{\infty}t^{-p^n} \in\mathbb F_p[[1/t]]
$$ 
is algebraic over $\mathbb F_p(t)$ with degree $p$, while $\mu(f)=p$.  Osgood \cite{Os} and then 
 Lasjaunias and de Mathan \cite{LaDeM} obtained an improvement of the Liouville bound (namely the Thue bound) 
 for a large class of algebraic functions.  However not much is known about the irrationality exponent of algebraic functions in 
 $\mathbb F_p((1/t))$. For instance, it seems that we do not know whether $\mu(f)=2$ for almost every\footnote{This could mean something like,   
 among algebraic 
 Laurent series $f$ of degree and height at most  $M$, the proportion of those with $\mu(f)=2$ tends to one as $M$ tends to infinity.}  
 algebraic Laurent series in $\mathbb F_p((1/t))$. We also do not know what the set $\mathcal E$ of possible values taken by $\mu(f)$ is precisely when 
 $f$ runs over the algebraic Laurent series.

 In this direction, we mention that it is possible to use an analogue of Theorem Sh2 for power series with coefficients in a finite field 
 (the proof of which is identical). Thakur \cite{Th99} uses such a result in order to exhibit explicit power series 
 $f(t)\in\mathbb F_p[[1/t]]$  with any prescribed 
irrationality measure $\nu\geq 2$, with $\nu$ rational.   In other words, this proves that $\mathbb Q_{\geq 2}\subset \mathcal E$, where 
$\mathbb Q_{\geq 2}:=\mathbb Q\cap [2,+\infty)$. 
These  power series are defined as linear combinations of Mahler-type series which shows that they are algebraic, while 
the analogue of Theorem Sh2 allows us to describe their continued fraction expansion and thus to easily compute the value of $\mu(f)$, as previously.     
Note that this result can also be obtained by considering only continued fractions, as shown independently by Thakur \cite{Th99} and 
Schmidt \cite{Schmidt00}.  It is expected, but not yet proved, that $\mathcal E=\mathbb Q_{\geq 2}$. 
For a recent survey about these questions, we refer the reader to  \cite{Th09}.

\section{Approximation by quadratic numbers}\label{sec: cf1}

A famous consequence of the subspace theorem provides a natural analogue of Roth's theorem in which 
rational approximations are replaced by quadratic ones. More precisely,  
if $\xi$ is an algebraic number of degree at least $3$ and $\varepsilon$ is a positive real number, then the inequality  
\begin{equation}\label{eq: sch}
\left\vert \xi - \alpha\right\vert < \frac{1}{H(\alpha)^{3+\varepsilon}}  ,
\end{equation}
has only finitely many quadratic solutions $\alpha$.  Here $H(\alpha)$ denotes the (naive) height of $\alpha$, 
that is, the maximum of the modulus of 
the coefficients of its minimal polynomial.

In this section, we give our fourth proof of transcendence for $\kappa$ which is obtained as a consequence of Theorem Sh1 
(see Section \ref{section: folding}) 
and Theorem AB2 stated below.   
We observe that some repetitive patterns occur in the continued fraction expansion of $2\kappa$ and then we use them to find infinitely many 
good quadratic approximations $\alpha_n$ to $2\kappa$.  However, a more careful analysis  would show that 
$$
\left\vert 2\kappa - \alpha_n \right\vert \gg \frac{1}{H(\alpha_n)^3}  ,
$$
so that we cannot directly apply (\ref{eq: sch}). Fortunately, the subspace theorem offers a lot of  freedom and 
adding some information about the minimal polynomial of our approximations finally allows us to conclude.

We keep the notation from Sections \ref{section: roth} and \ref{section: folding}.  
Let ${\bf a}=a_1a_2\cdots$ 
be an infinite word and $w\geq 1$ be a real number. 
We say that ${\bf a}$ 
satisfies Condition $(*)_w$ 
if there exists 
a sequence of finite words $(V_n)_{n \ge 1}$ such that the following hold.

\begin{itemize}

\item[{\rm (i)}] For any $n \ge 1$, the word $V_n^w$ is a prefix
of the word ${\bf a}$.

\item[{\rm (ii)}] The sequence $(\vert V_n\vert)_{n \ge 1}$ is 
increasing.
\end{itemize}

The following result is a special instance of Theorem 1 in \cite{AdBuActa}.

\begin{thmAB2}
Let $(a_n)_{n \ge 1}$ be a bounded sequence of positive integers such that ${\bf a}=a_1a_2\cdots$ 
satisfies Condition $(*)_w$ for some real number $w>1$.
Then the real number
$$
\xi:= [0, a_1, a_2,  \ldots]
$$ 
is either quadratic or transcendental.
\end{thmAB2}

We only outline the main idea of the proof and refer the reader to \cite{AdBuActa} for more details. 

\begin{proof}[Hint of proof]
Assume that the parameter $w > 1$ is fixed, as well as the 
sequence $(V_n)_{n \ge 1}$ occurring in the
definition of Condition $(*)_w$. 
Set also $s_n:=\vert V_n\vert$. 
We want to prove that the real number
$$
\xi:= [0, a_1, a_2, \ldots]
$$ 
is either quadratic or transcendental. We assume that $\alpha$ is algebraic of degree at
least three and we aim at deriving a contradiction.

Let $p_n/q_n$ denote the $n$th convergent  
to $\xi$. 
The key fact for the proof  is the observation
that $\xi$ has infinitely many good quadratic approximations 
obtained by truncating its continued fraction
expansion and completing by periodicity. 
Let $n$ be a positive integer and let us define the quadratic number $\alpha_n$ 
as having the following purely periodic 
continued fraction expansion: 
$$
\alpha_n:= [0, V_nV_nV_n \cdots ]  .
$$
Then
$$
\left\vert \xi -\alpha_n\right\vert < \frac{1}{q_{\lfloor ws_n\rfloor^2}}  ,
$$
since by assumption the first $\lfloor ws_n\rfloor$ partial quotients of $\xi$ and $\alpha_n$ are the same.
Now observe that $\alpha_n$ is a root of the quadratic polynomial
$$
P_n (X) := q_{s_n-1} X^2 + (q_{s_n} - p_{s_n-1}) X - p_{s_n}  .
$$
By Rolle's theorem, 
we have
$$
\vert P_n (\xi)\vert = \vert P_n (\xi) - P_n (\alpha_n)\vert 
\ll  q_{s_n}  |\xi - \alpha_n| \ll  \frac{q_{s_n} }{ q_{\lfloor w s_n\rfloor}^2}  \cdot
$$
Furthermore, by the theory of continued fractions we also have  
$$
\vert q_{s_n} \xi - p_{s_n}\vert  \le \frac{1}{q_{s_n}}   \cdot
$$
Consider the four linearly independent linear forms with real algebraic coefficients:
$$
\begin{array}{lll}
L_1(X_1, X_2, X_3, X_4)& = & \xi^2 X_2 + \xi (X_1 - X_4) - X_3, \\ 
L_2(X_1, X_2, X_3, X_4) &= & \xi X_1 - X_3, \\ 
L_3(X_1, X_2, X_3, X_4) &= & X_1, \\ 
L_4(X_1, X_2, X_3, X_4) &= & X_2. 
\end{array}
$$
Evaluating them on the integer quadruple 
$(q_{s_n}, q_{s_n-1}, p_{s_n}, p_{s_n-1})$, a simple computation using continuants 
shows that
$$
\prod_{1 \le j \le 4}  |L_j (q_{s_n}, q_{s_n-1}, p_{s_n}, p_{s_n-1})|
\ll q_{s_n}^2  q_{\lfloor w s_n\rfloor}^{-2} < \frac{1}{q_{s_n}^{\varepsilon}}  ,
$$
 for some positive number $\varepsilon$, when $n$ is large enough. 
We can thus apply the subspace theorem. We obtain  
that all the integer points $(q_{s_n}, q_{s_n-1}, p_{s_n}, p_{s_n-1})$, $n\in\mathbb N$, 
belong to  a finite number of proper subspaces of $\mathbb Q^4$. 
After some work, it can be shown that this is possible only if $\xi$ is quadratic, a contradiction.  
\end{proof}

We are now ready to give a new proof of transcendence for $\kappa$.

\begin{proof}[Fourth proof] 
Shallit \cite{Sh1} proves that the continued fraction 
of $\kappa$ is not ultimately periodic.  Here we include, for the sake of completeness, a similar proof that the word $A_{\infty}$ is not ultimately periodic. 
Set $A_{\infty}:=a_1a_2\cdots$.  
We argue by contradiction assuming that $A_{\infty}$ is ultimately periodic. 
There thus exist two positive integers $r$ and $n_0$ such that 
\begin{equation}\label{eq: per}
a_{n+jr} = a_{n}  ,
\end{equation} 
for every $n\geq n_0$ and $j\geq 1$.

For every positive integer $i$, set $k_i:= \vert A_i\vert = 10\cdot 2^{i-1}$.  
Let us fix a positive integer $n$ such that $k_n \geq r+ n_0$.   
Theorem Sh1 implies that $A_{\infty}$ begins with
$$
A_{n+1} = a_1\cdots a_{k_n-2}a_{k_n-1}a_{k_n}  12  a_{k_n-2}\cdots a_1 
$$
and since $A_n$ always ends with $11$ (see (\ref{eq: 11})), we obtain that 
\begin{equation}\label{eq: 1112}
   A_{n+1}=a_1\cdots a_{k_n-2} 11 12  a_{k_n-2}\cdots a_1  .
\end{equation}
We thus have
\begin{equation}\label{eq: sym}
a_{k_n+x+1} = a_{k_n-x}  ,
\end{equation}
for every integer $x$ with $2\leq x \leq k_n-1$. 
Since the word $1112$ occurs infinitely often 
in $A_{\infty}$, we see that $r\geq 4$.  We can thus choose $x=r-2$ in (\ref{eq: sym}), 
which gives that
$$
a_{k_n+r-1} = a_{k_n-r+2}  .
$$
Then (\ref{eq: per}) implies that 
$$
a_{k_n+r-1}  = a_{k_n-1} 
$$
and 
$$
a_{k_n-r+2} =a_{k_n +2} . 
$$
Finally, we get  that 
$$
a_{k_n-1}  = a_{k_n +2} ,
$$
that is 
$$
1=2 ,
$$
a contradiction.  Hence, $A_{\infty}$ is not ultimately periodic.

Now, set
$$
\xi := \frac{4\kappa-3}{2-2\kappa}  \cdot
$$
Clearly it is enough to prove that $\xi$ is transcendental.  
By Theorem Sh1, we have that 
$$
2\kappa =[1,A_{\infty}] =  [1,a_1,a_2,\ldots] \
$$
and a simple computation shows that 
$$
\xi = [0,a_3,a_4,\cdots]  .
$$
Let us show that the infinite word $a_3a_3\cdots$ satisfies Condition $(*)_{1+3/10}$.  
Let $n$ be a positive integer.  
Using the definition of  $A_{n+2}$, we infer from  (\ref{eq: 1112}) that  
$$
A_{n+2} = a_1\cdots a_{k_n-2}  1112  a_{k_n-2}\cdots a_1  12  a_3a_4\cdots  a_{k_n-2} 2 111a_{k_n-2}\cdots a_1
$$
Setting  $V_n := a_3a_4\cdots  a_{k_n-2} 1112  a_{k_n-2}\cdots a_112$, we obtain that $A_{n+2}$ begins with 
$$
a_1a_2 V_n^{1+ (k_n-4)/2k_n}  .
$$
Since $k_n\geq 10$, we thus deduce that $A_{\infty}$ begins with $a_1a_2V_n^{1+3/10}$, for every integer $n\geq 1$. 
This shows that the infinite word $a_3a_4\cdots$ 
satisfies Condition $(*)_{1+3/10}$.  
Applying Theorem AB2, we obtain that $\xi$ is either quadratic or transcendental. However, we infer from 
Lagrange's theorem that $\xi$ cannot be quadratic, for we just have shown that the word $a_3a_4\cdots$ is not ultimately periodic. 
Thus $\xi$ is transcendental, concluding the proof. 
\end{proof}

\subsection{Beyond this proof} 
Very little is known regarding the size of the partial quotients
of algebraic real numbers of degree at least three. From  
numerical evidence and a belief that these numbers behave 
like most of the numbers in this respect, 
it is often conjectured that their partial quotients form an unbounded 
sequence.  
Apparently, Khintchine \cite{Kh} 
was the first to consider this question 
(see \cite{All00,Sha,MIW}
for surveys including a discussion of this problem). Although almost 
nothing has been proved yet in this direction, Lang \cite{LangSMF,Lang} made some more general 
speculations, including the fact that 
algebraic numbers of degree at least three should behave like most 
of the numbers with respect to the Gauss--Khintchine--Kuzmin--L\'evy laws. 
This conjectural picture is thus very similar to the one encountered in Section \ref{section: roth} 
regarding the expansion of algebraic irrational numbers in integer bases.

As a first step, it is worth proving that 
real numbers with a ``too simple" continued fraction expansion
are either quadratic or transcendental.  Of course, the term ``simple" can  lead to many interpretations. 
It may, for instance, denote real numbers whose continued fraction expansion can be produced in a simple
algorithmical, dynamical, combinatorial or arithmetical way. 
In any case, it is reasonable to expect that such expansions should be close to be periodic.  
In this direction, there is a long tradition of using 
an excess of periodicity to prove the transcendence of 
some continued fractions (see, for instance, \cite{ADQZ,Baker62,Bax,Dav1,Dav2,Dav3,Maillet,Quef1,Quef2}).  
Adamczewski and Bugeaud  \cite{AdBuActa} use the freedom offered by the subspace theorem 
to prove some combinatorial
transcendence criteria, which provide significant improvements of those previously obtained in \cite{ADQZ,Dav1,Dav2,Quef1,Quef2}.  
Some transcendence measures for such continued fractions  are also recently given in \cite{AdBuJEMS,Bu3}, following the general approach 
developed in \cite{AdBuPLMS} (see also \cite{AdBuCrelle} for similar results related to integer base expansions).

 Theorem AB2 gives the transcendence of non-quadratic real numbers $\xi$  
whose continued fraction expansion begins with arbitrarily  long blocks of partial quotients of the form $V_n^{1+\varepsilon}$, for some positive 
$\varepsilon$. 
The key fact in the proof to apply the subspace theorem is to see that 
the linear form 
$$
 \xi^2 X_2 + \xi (X_1 - X_4) - X_3
$$ 
takes small values at the integer quadruples $(q_{s_n}, q_{s_n-1}, p_{s_n}, p_{s_n-1})$, 
where $s_n=\vert V_n\vert$ and where $p_n/q_n$ denotes the $n$th convergent to $\xi$. 
This idea can be naturally generalized to the important case where the repetitive patterns do not occur at the very beginning of the expansion. 
Indeed, if the continued fraction expansion of $\xi$ begins with a block of partial quotients of the form $U_nV_n^{1+\varepsilon}$, 
then $\xi$ is well approximated by the quadratic number $\alpha'_n:=[0,U_nV_nV_nV_n\cdots]$.  
As shown in \cite{AdBuActa}, when dealing 
with these more general patterns, we can argue similarly, but we have now to estimate 
the linear form 
$$
L(X_1,X_2,X_3,X_4):= \xi^2X_1-\xi(X_2+X_3)+X_4 
$$
at the integer quadruples  
$$\begin{array}{c}
{\bf x}_n:= (q_{r_n-1}q_{r_n+s_n}-q_{r_n}q_{r_n+s_n}, q_{r_n-1}p_{r_n+s_n}-q_{r_n}p_{r_n+s_n-1}, \hspace{3cm}\\
\hspace{2cm}p_{r_n-1}q_{r_n+s_n}-p_{r_n}q_{r_n+s_n-1}, p_{r_n-1}p_{r_n+s_n}-p_{r_n}p_{r_n+s_n-1})  ,
\end{array}
$$
where $r_n:=\vert U_n\vert$, $s_n:=\vert V_n\vert$, and $p_n/q_n$ denotes the $n$th convergent to $\xi$. 
In a recent paper, Bugeaud \cite{Bu12} remarks that the quantity $\vert L({\bf x}_n)\vert$ was overestimated in the proof given in  
\cite{AdBuActa}\footnote{The authors of \cite{AdBuActa} use that 
$\vert L({\bf x}_n)\vert  \ll q_{r_n} q_{r_n+s_n} q_{ r_n \lfloor (1+\varepsilon) s_n \rfloor }^{-2}$ while it is actually elementary to see that 
$\vert L({\bf x}_n)\vert  \ll q_{r_n}^{-1} q_{r_n+s_n} q_{ r_n \lfloor (1+\varepsilon) s_n \rfloor }^{-2}$.}. 
Taking into account this observation, the method developed in \cite{AdBuActa} has much striking consequences than those 
initially announced. As an illustration, we now have that the continued fraction expansion of an algebraic number of 
degree at least $3$ cannot be generated by a finite automaton.  

\section{Simultaneous approximation by rational numbers}\label{sec: cf2}

Another classical feature of the subspace theorem is that it can also deal with simultaneous approximation 
of several real numbers by rationals. Our last proof of the transcendence of $\kappa$ relies on this principle. 
We use the occurrences of some symmetric patterns in the continued fraction expansion of $2\kappa-1$ 
to find good simultaneous rational approximations of $2\kappa-1$ and $(2\kappa-1)^2$.

\begin{proof}[Fifth proof]
We keep the notation of Sections \ref{section: folding} and \ref{sec: cf1}. 
Set $\xi:=(2\kappa-1)$. Clearly it is enough to prove that $\xi$ is transcendental.  
We argue by contradiction assuming that $\xi$ is algebraic.

By Theorem Sh1, we have that  $\xi=[0,A_{\infty}]= [0,a_1,a_2,\ldots]$.  
We let $p_n/q_n$ denote the $n$th convergent to $\xi$. 
We also set as previously $k_n:=\vert A_n\vert$. 
By the theory of 
continued fraction, we have 
$$
\left\vert \xi - \frac{ p_{k_n-1} }{ q_{k_n-1} } \right\vert < \frac{1}{q_{k_n-1}^2}  
\mbox{ and } 
\left\vert \xi - \frac{ p_{k_n} }{ q_{k_n} } \right\vert < \frac{1}{q_{k_n}^2}  \cdot
$$
Let us also recall the so-called mirror formula (see, for instance, \cite{AdAl}): 
\begin{equation}\label{eq: mirror}
\frac{q_{n-1}}{q_{n}} =[0,a_n,\ldots,a_1]  .
\end{equation}
Since $A_{n} = B_{n-1}  1112  (B_{n-1})^R$, we have  
$$
\frac{q_{k_n-1}}{q_{k_n}} =[0,B_{n-1},2,1,1,1, (B_{n-1})^R] 
$$
and a simple computation using continuants shows that\footnote{Indeed, the theory of continuants implies that the denominator of the rational number $[0,B_{n-1}]$, 
say $r_n$, grows roughly $\sqrt{q_{k_n}}$ and thus $r_n^{-2}$ essentially behaves like $q_{k_n}^{-1}$. For more details about continuants, see, for instance, \cite[Sec.\,8.2]{AdBuCant},
\cite[Sec.\ 5]{AdBuCrelle1},
or \cite[Sec.\,3]{AdBuJLMS}. } 
$$
\left\vert \xi - \frac{q_{k_n-1}}{q_{k_n}} \right\vert \ll \frac{1}{q_{k_n}} 
 \mbox{ and } 
\left\vert \frac{p_{k_n}}{q_{k_n}} - \frac{q_{k_n-1}}{q_{k_n}} \right\vert \ll \frac{1}{q_{k_n}}   \cdot
$$
Then we have 
$$
\begin{array}{lll}
\left\vert \xi^2 - \frac{ p_{k_n-1} }{ q_{k_n} } \right\vert &=& 
\left\vert \xi^2 - \frac{ p_{k_n-1} }{ q_{k_n-1} } \cdot \frac{ q_{k_n-1} }{ q_{k_n} } \right\vert \\ \\
&=& \left\vert \left (\xi-\frac{ p_{k_n-1} }{ q_{k_n-1} }\right)\left(\xi+\frac{ q_{k_n-1} }{ q_{k_n} } \right) 
+ \xi\left(\frac{ p_{k_n-1} }{ q_{k_n-1} }- \frac{ q_{k_n-1} }{ q_{k_n} } \right)\right\vert \\ \\
& \ll&  q_{k_n}^{-1}  \cdot
 \end{array}
$$

Consider the four linearly independent linear forms with real algebraic coefficients:
$$
\begin{array}{lll}
L_1(X_1, X_2, X_3, X_4)& = & \xi^2 X_1 - X_4, \\ 
L_2(X_1, X_2, X_3, X_4) &= & \xi X_1 - X_3, \\ 
L_3(X_1, X_2, X_3, X_4) &= & \xi X_2-X_4, \\ 
L_4(X_1, X_2, X_3, X_4) &= & X_2. 
\end{array}
$$
Evaluating them on the integer quadruple 
$(q_{k_n}, q_{k_n-1}, p_{k_n}, p_{k_n-1})$, our previous 
estimates implies that 
$$
\prod_{1 \le j \le 4}  \left\vert L_j (q_{k_n}, q_{k_n-1}, p_{k_n}, p_{k_n-1})\right\vert 
\ll  \frac{1}{q_{k_n}}  \cdot
$$
The subspace theorem thus implies   
that all the integer points $(q_{k_n}, q_{k_n-1}, p_{k_n}, p_{k_n-1})$, $n\geq 1$, 
belong to  a finite number of proper subspaces of $\mathbb Q^4$.  
Thus, there exists a nonzero integer quadruple $(x,y,z,t)$ such that 
$$
q_{k_n}x+  q_{k_n-1}y+ p_{k_n}z+ p_{k_n-1}t = 0 ,
$$
for all $n$ in an infinite set of positive integers $\mathcal N$. 
Dividing by $q_{k_n}$ and letting $n$ tend to infinity along $\mathcal N$, we obtain that 
$$
x+ \xi(y+z) + \xi^2 t =0 .
$$
Since $A_{\infty}$ is not ultimately periodic (see Section \ref{sec: cf1}), it follows from Lagrange's theorem 
that $\xi$ is not quadratic and thus  $x=t=(y+z)=0$. 
This gives that $q_{k_n-1} = p_{k_n}$ for every $n\in\mathcal N$.  
But
$$
\frac{q_{k_n-1} }{q_{k_n}} = [0,B_n,2,1,1,1, (B_n)^R] 
\not= [0,B_n,1,1,1,2, (B_n)^R] = \frac{p_{k_n} }{q_{k_n}}  ,
$$
a contradiction.
\end{proof}

\subsection{Beyond this proof} 
As already mentioned in Section \ref{sec: cf1}, there is a long tradition in using 
an excess of periodicity to prove the transcendence of 
some continued fractions. 
The fact that occurrences of symmetric patterns can actually 
give rise to transcendence statements is more surprising. 
The connection between palindromes\footnote{Recall that a palindrome 
is a word $W=a_1\cdots a_n$ that is equal to its reversal $W^R=a_n\cdots a_1$. Thus 
a palindrome may be considered as a perfect symmetric pattern.} 
in continued fractions and simultaneous approximation of a number and its square 
is reminiscent of works of Roy about extremal numbers \cite{Roy03bis,Roy03,Roy04} (see also \cite{Fi}).  
Inspired by this original discovery of Roy, Adamczewski and Bugeaud  \cite{AdBuFourier}  use 
 the subspace theorem and the mirror formula (\ref{eq: mirror}) 
to establish several combinatorial transcendence criteria for continued fractions involving symmetric patterns  
(see also \cite{Bu12} for a recent improvement of one of these criteria).    
The same authors  \cite{AdBuJEMS} also give some transcendence measures for such continued fractions.   
 As a simple illustration, they prove in  \cite{AdBuFourier} that if the continued fraction of 
 a real number  begins with arbitrarily large palindromes, then this number is either quadratic or transcendental.

 It is amusing to mention that in contrast there are only very partial results 
 about the transcendence of numbers whose decimal expansion involves 
 symmetric patterns (see \cite{AdBuIMRN}).  In particular, it is not known whether a real number whose decimal expansion begins with arbitrarily 
 large palindromes is either rational or transcendental. 
 
 Beyond the study of extremal numbers and transcendence results, Adamczewski and Bugeaud  \cite{AdBuJLMS} use continued fractions with 
 symmetric patterns  to provide explicit examples for the famous Littlewood conjecture on simultaneous Diophantine approximation.
 
\section{Acknowledgments}
 
 This work was supported by the project Hamot, ANR 2010 BLAN-0115-01.  
The idea to make this survey came up with a talk  I gave for the conference 
\emph{Diophantine Analysis and Related Fields} that was held in Tokyo in March 2009. I would like to take the opportunity to thank 
the organizers of this conference and especially Noriko Hirata-Kohno. I would also like to thank  the anonymous referees for their useful comments. 
My interest for transcendence grew quickly after I followed some lectures by Jean-Paul Allouche when I was a 
Ph.D. student.  I am deeply indebted to Jean-Paul for giving such stimulating lectures!

\end{document}